\documentclass[10pt,twoside, reqno]{amsart}

\usepackage{amsthm,amsmath,natbib,amssymb,tikz, rotating}
\usetikzlibrary{calc,fadings,decorations.pathreplacing,intersections,arrows}
\usepackage[english]{babel}
\usepackage[T1]{fontenc}
\usepackage[colorlinks,citecolor=blue,urlcolor=blue]{hyperref}

\begin{document}

\newtheorem{thm}{Theorem}
\newtheorem{lem}{Lemma}
\newtheorem{cor}{Corollary}
\newtheorem{prop}{Proposition}
\theoremstyle{definition}
\newtheorem{dfn}{Definition}
\newtheorem{alg}{Algorithm}
\newtheorem{rem}{Remark}
\newtheorem{ex}{Example}
\renewcommand{\phi}{\varphi}
\renewcommand{\epsilon}{\varepsilon}

\setlength{\textheight}{19.5 cm}
\setlength{\textwidth}{12.5 cm}
\setlength{\parindent}{0pt}
\setlength{\parskip}{1.3ex}

\title[SURE and reduced rank estimation]{A comment on Stein's unbiased
  risk estimate for reduced rank estimators}

\author[N. R. Hansen]{Niels Richard Hansen}
\address{Department of Mathematical Sciences,
University of Copenhagen,
Universitetsparken 5.
2100 Copenhagen \O.
Denmark, Niels.R.Hansen@math.ku.dk}

\subjclass[2010]{}


\maketitle

\begin{abstract}
  In the framework of matrix valued observables with low rank means,
  Stein's unbiased risk estimate (SURE) can be useful for risk
  estimation and for tuning the amount of shrinkage towards low rank
  matrices. This was demonstrated by \cite{Candes:2013} for singular
  value soft thresholding, which is a Lipschitz continuous
  estimator. SURE provides an unbiased risk estimate for an estimator
  whenever the differentiability requirements for Stein's lemma are
  satisfied. Lipschitz continuity of the estimator is sufficient, but
  it is emphasized that differentiability Lebesgue almost everywhere
  isn't. The reduced rank estimator, which gives the best
  approximation of the observation with a fixed rank, is an example of
  a discontinuous estimator for which Stein's lemma actually
  applies. This was observed by \cite{Mukherjee:2015}, but the proof
  was incomplete.  This brief note gives a sufficient condition for
  Stein's lemma to hold for estimators with discontinuities, which is
  then shown to be fulfilled for a class of spectral function
  estimators including the reduced rank estimator. Singular value hard
  thresholding does, however, not satisfy the condition, and Stein's
  lemma does not apply to this estimator.
\end{abstract}

\section{Introduction}

Let $\mathbf{Y}$ be an $p \times q$ matrix with $p \geq q$ and
singular value decomposition
$$\mathbf{Y} = \sum_{k=1}^q d_k u_k v_k^T.$$
If all the singular values are unique and ordered as $d_1 > d_2 > \ldots > d_q
\geq 0$ the rank $r$ approximation of $\mathbf{Y}$ that minimizes the
Frobenius norm error is given by
$$\hat{\mu}(r) = \sum_{k=1}^r d_k u_k v_k^T,$$
for $r \in \{1, \ldots, q - 1\}$. The hard threshold approximation
given by
$$\overline{\mu}(\lambda) = \sum_{k=1}^q d_k 1(d_k \geq \lambda) u_k v_k^T$$
for $\lambda \geq 0$ yields the same sequence of approximations but parametrized
differently.  

If $\mathbf{Y} = (Y_{ij})_{i,j}$ has independent entries with 
$$Y_{ij} \sim \mathcal{N}(\mu_{ij}, \sigma^2)$$
and $\mu = (\mu_{ij})_{i,j}$ is of low rank, the approximations $\hat{\mu}(r)$ or
$\overline{\mu}(\lambda)$ are sensible estimators of $\mu$ for suitable
choices of $r$ or $\lambda$. Both estimators are examples from 
the more general class of \emph{spectral function} estimators
$$\hat{\mu} =  \sum_{k=1}^q f_k(d_k) u_k v_k^T, \quad \lambda_1 >
\lambda_2 > \ldots > \lambda_q \geq 0$$
for some spectral functions $f_k : [0, \infty) \to [0, \infty)$. The
hard thresholding estimator has $f_k(d) = d1(d \geq \lambda)$, and 
the estimator with $f_k(d) = d1(k \leq r)$, which gives the best rank $r$
approximation, will be referred to as the reduced rank estimator. 

In the framework of spectral function estimators, \cite{Candes:2013}
derive an explicit formula (formula (9) in their paper) for the
divergence of the estimator when $\mathbf{Y}$ has distinct singular
values and $f_k$ is differentiable in a neighborhood of $d_k$. 
This divergence is required for the computation of SURE, and
it's therefore important for applications. They
also demonstrate in detail (Lemma III.3) how Stein's lemma applies in
the special case of \emph{soft thresholding}, which is a spectral
function estimator with all spectral functions equal to the continuous
function $f_k(d) = (d - \lambda)_+$.  \cite{Candes:2013} further show
that the divergence extends
continuously to all matrices (Theorem IV.6) whenever all spectral
functions are identical and sufficiently smooth. That result doesn't apply to the hard thresholding
estimator (though the spectral functions are identical, they are
discontinuous) or to the reduced rank estimator (the spectral functions
are not all identical). 

\cite{Mukherjee:2015} derive similar formulas for the divergence --
apparently unaware of the paper by \cite{Candes:2013}. One difference
is that \cite{Mukherjee:2015} focus on the regression setup, where the
columns of the observation matrix are projected onto a fixed
subspace before it is subjected to a low rank approximation. 

Neither \cite{Candes:2013} nor \cite{Mukherjee:2015} provides 
conditions for general spectral function estimators that ensure that
Stein's lemma applies. \cite{Mukherjee:2015} indicate on page 460 that the mere existence of
the partial derivatives (Lebesgue) almost everywhere is sufficient for
Stein's lemma, which is not the case as shown below.
\cite{Candes:2013} state a correct version of Stein's lemma as
their Proposition III.1, which (correctly) assumes weak differentiability of the
estimator. Weak differentiability of the soft threshold estimator is
then shown in detail (using that it's Lipschitz), but  \cite{Candes:2013} don't
clarify if e.g. their Theorem IV.6 implies the required weak
differentiability for more general spectral function estimators. The
theorem does not cover the reduced rank estimator anyway.

The purpose of this note is to provide conditions  ensuring that a 
spectral function estimator is, indeed,
weakly differentiable so that Stein's lemma applies. In particular, we
show that the reduced rank estimator is weakly differentiable so that
the SURE formulas as given by  \cite{Candes:2013} or
\cite{Mukherjee:2015} result in unbiased estimation of the risk. To
illustrate the relevance of such sufficient conditions, we show by a
small simulation that the SURE formula for singular value hard
thresholding doesn't give unbiased estimation of the risk.

\section{Stein's lemma}

In this section we state a sufficient condition for Stein's lemma,
which can then be shown to hold for the reduced rank estimator. It's
formulated for $n$-dimensional Gaussian vectors and applies to
the matrix valued observations and estimators above by taking $n = pq$. 

Let $y \sim N(\mu, \sigma^2
I)$ be an $n$-dimensional Gaussian random variable and $\hat{\mu}$
an estimator of $\mu$ with finite second moment; $E||\hat{\mu}||_2^2 < \infty.$
For such an estimator, Stein's lemma or Stein's identity (Lemma 1 or Lemma 2 in \cite{Stein:1981})
implies that 
\begin{equation} 
\label{eq:stein}
\frac{1}{\sigma^2} \sum_{i=1}^n
\mathrm{cov}(\hat{\mu}_i, y_i)  = \sum_{i=1}^n E \left( \partial_i
    \hat{\mu}_i \right) = E \left(\nabla \cdot \hat{\mu}\right)
\end{equation}
provided that $\hat{\mu}$ is \emph{almost differentiable} and 
\begin{equation}
\label{eq:momcond}
\sum_{i=1}^n E \left| \partial_i\hat{\mu}_i\right| < \infty.
\end{equation}
The quantity
\begin{equation}
\label{eq:df}
\mathrm{df} = \frac{1}{\sigma^2} \sum_{i=1}^n
\mathrm{cov}(\hat{\mu}_i, y_i) 
\end{equation}
is usually referred to as the effective degrees of freedom, and when Stein's lemma holds, the divergence
$$ \nabla \cdot \hat{\mu} = \sum_{i=1}^n \partial_i \hat{\mu}_i$$
is an unbiased estimate of $\mathrm{df}$. However, if $\hat{\mu}$ is
not continuously differentiable everywhere, care has to be taken to
verify almost differentiability even if the divergence
$\nabla \cdot \hat{\mu}$ is well defined Lebesgue almost
everywhere. Likewise, the moment condition \eqref{eq:momcond} on the
derivative must be verified. All conditions for Stein's lemma are
fulfilled if $\hat{\mu}$ is Lipschitz as shown e.g. by
\cite{Candes:2013} in their Lemma III.2, but for discontinuous
estimators the situation is more complicated. 

For example, take $\hat{\mu} = 1(y \geq 0)$ for $n = 1$, then it is straightforward to
see that \eqref{eq:stein} doesn't hold, though this function is
continuously differentiable everywhere except in $0$. A ``real''
estimator for which Stein's lemma doesn't apply is the hard threshold estimator,
$\hat{\mu} = y1(|y| \geq c)$, as treated extensively by
\cite{Tibshirani:2015}, and further examples and an extension of Stein's
lemma are given by \cite{Mikkelsen:2017}. Differentiability Lebesgue almost everywhere
is generally not sufficient for Stein's lemma to apply and
\eqref{eq:stein} to hold.

We state below a sufficient condition for \eqref{eq:stein} to hold
that can be verified in the case of reduced rank estimation. To this
end let $\mathcal{H}^{n-1}$ denote the  $(n-1)$-dimensional Hausdorff measure.

\pagebreak

\begin{prop} \label{prop:hausdorff}  Let $E \subseteq
  \mathbb{R}^n$ be a closed set such that  
$$\hat{\mu} : E^c \to \mathbb{R}^n$$ 
is continuously differentiable, then \eqref{eq:stein} holds if 
$\mathcal{H}^{n-1}(E) = 0$ and either \eqref{eq:momcond} is satisfied or $\nabla
\cdot \hat{\mu} \geq 0$ Lebesgue almost everywhere. 
\end{prop}

\begin{proof} We first show that the condition $\mathcal{H}^{n-1}(E) =
  0$ implies almost differentiability of $\hat{\mu}$. 
As pointed out by \cite{Johnstone:1988}, see also
\cite{fourdrinier:2012} and \cite{Candes:2013}, the differentiability requirement for Stein's lemma is
effectively that the estimator is weakly differentiable, that is, that
it belongs to the Sobolev space
$W^{1,1}_{\mathrm{loc}}(\mathbb{R}^n)$. This is the case if for all
$i$ and Lebesgue almost all $(y_1, \ldots, y_{n-1}) \in \mathbb{R}^{n-1}$
the function
\begin{equation} \label{eq:absolute}
t \mapsto \hat{\mu}(y_1, \ldots, y_{i-1}, t, y_i, \ldots,  y_{n-1})
\end{equation}
is absolutely continuous on compact intervals, see Theorem 4.9.2.2 in
\cite{Evans:1992}. Alternatively, the proofs in
\cite{Stein:1981} show directly that almost everywhere absolute continuity
of the maps \eqref{eq:absolute} implies \eqref{eq:stein}. See also
Definition 2 and Lemma 5 in \cite{Tibshirani:2015}. When $\mathcal{H}^{n-1}(E) = 0$ the projection maps 
$$(y_1, \ldots, y_{i-1}, y_i, y_{i + 1}, \ldots,  y_n) \mapsto (y_1, \ldots,
y_{i-1}, y_{i + 1}, \ldots,  y_{n-1})$$
for $i = 1, \ldots, n$ map $E$ onto Lebesgue null sets by Corollary
2.4.1.1 in \cite{Evans:1992}. Thus when $\hat{\mu}$ is continuously
differentiable on $E^c$, \eqref{eq:absolute} is continuously
differentiable, and thus absolutely continuous, Lebesgue almost
everywhere. 

Having established almost differentiability,  \eqref{eq:momcond}
then implies \eqref{eq:stein}. In the second half of the proof we show that $\nabla
\cdot \hat{\mu} \geq 0$ is an alternative sufficient condition. To
this end, note that since the estimator belongs to
$W^{1,1}_{\mathrm{loc}}(\mathbb{R}^n)$ then, in fact, 
$$\sum_{i=1}^n \int \hat{\mu}_i(y) \partial_i \psi(y) \mathrm{d} y = -
\int \nabla \cdot \hat{\mu}(y) \ \psi(y) \mathrm{d} y$$
for all $\psi \in C_c^{\infty}(\mathbb{R}^n)$.

Let $\phi$ denote the density for the $\mathcal{N}(\mu, \sigma^2 I)$
distribution. Choose $\kappa \in C_c^{\infty}(\mathbb{R}^n)$ with
$\kappa(y) \in [0,1]$ and such that $\kappa(y) = 1$ for
$||y||_2 \leq 1$. Defining
$$\phi_n(y) = \phi(y) \kappa(n^{-1} y),$$ 
then $\phi_n \in C_c^{\infty}(\mathbb{R}^n)$ and 
$$\phi(y) \geq \phi_n(y) \geq \phi(y) 1(||y||_2 \leq n) \nearrow \phi(y)$$
for $n \to 0$. We also observe that 
$$\partial_i \phi_n(y) = (\partial_i \phi(y)) \kappa(n^{-1} y) +
\phi(y) n^{-1} \partial_i \kappa(n^{-1} y) \rightarrow \partial_i
\phi(y)$$ 
for $n \to \infty$. Here we used that $\partial_i \kappa(n^{-1} y) =
0$ for $||y||_2 < n$ and $\kappa(n^{-1} y) \to 1$ for $n \to
\infty$. Moreover, 
$$|\hat{\mu}_i(y) \partial_i \phi_n(y)| \leq |\hat{\mu}_i(y) (y_i -
\mu_i) \phi(y)| / \sigma^2 +
C |\hat{\mu}_i(y)| \phi(y)$$
for some constant $C$, and the right hand side above is
integrable. Thus by dominated convergence, 
$$\lim_{n \to \infty} \int \hat{\mu}_i(y) \partial_i \phi_n(y) \mathrm{d}
y =  \int \hat{\mu}_i(y) \partial_i \phi(y) \mathrm{d} y.$$ 
If $\nabla \cdot \hat{\mu}$ is almost everywhere
positive, it finally follows by monotone convergence combined with the
dominated convergence above that
\begin{align*}
E (\nabla \cdot \hat{\mu}) & = \int \nabla \cdot \hat{\mu} \ \phi(y)
\mathrm{d} y \\
& = \lim_{n \to \infty} \int \nabla \cdot \hat{\mu}(y) \
\phi_n(y) \mathrm{d} y \\
& = - \lim_{n \to \infty} \sum_{i=1}^n \int \hat{\mu}_i(y) 
\partial_i \phi_n(y) \mathrm{d} y \\
& = - \sum_{i=1}^n \int \hat{\mu}_i(y) \partial_i \phi(y) \mathrm{d} y \\
& = \frac{1}{\sigma^2} \sum_{i=1}^n \int \hat{\mu}_i(y) (y_i - \mu_i) \phi(y) \mathrm{d} y \\
& = \frac{1}{\sigma^2} \sum_{i=1}^n \mathrm{cov}(\hat{\mu}_i, y_i).
\end{align*}
\end{proof}

\section{Reduced rank estimators}

With the squared Frobenius norm as loss function, the risk
of an estimator $\hat{\mu}$ is 
$$E ||\hat{\mu} - \mu||_F^2 = E || \mathbf{Y} - \hat{\mu}||_F^2 -
\sigma^2 pq + 2 \sigma^2 \mathrm{df}.$$
When Stein's lemma applies, 
$$\mathrm{SURE} = || \mathbf{Y} - \hat{\mu}||_F^2 - \sigma^2 pq + 2
\sigma^2 \nabla \cdot \hat{\mu}$$
is an unbiased estimate of the risk. 

For spectral function estimators we observe that 
$$||\hat{\mu}||_F^2 = \sum_{k=1}^q f_k(d_k)^2.$$
Hence if $f_k(d) \leq C_k d$ for some constant $C_k$ (e.g. 
$f_k$ is shrinking the singular values as is the case for
reduced rank, and hard and soft thresholding), then 
$$||\hat{\mu}||_F^2 \leq \max_{k}\{ C_k^2\} \sum_{k=1}^q d_k^2
= C ||\mathbf{Y}||_F^2,$$
and $\hat{\mu}$ has finite second moment. 

\begin{thm} \label{thm:rr}
For the reduced rank estimator $\hat{\mu}(r)$, $\mathrm{SURE}$
is an unbiased estimate of the risk, and 
\begin{equation} 
\label{eq:nabrr}
\nabla \cdot \hat{\mu}(r) = pr + \sum_{k=1}^r \sum_{l= r + 1}^q
\frac{d_k^2 + d_l^2}{d_k^2 - d_l^2}
\end{equation}
whenever $\mathbf{Y}$ has $q$ distinct singular values.
\end{thm}

\begin{proof} The reduced rank estimator has finite second moment as
  argued above. We show that the conditions in Proposition
  \ref{prop:hausdorff} are fulfilled. To this end, note that it's
  clear from the explicit formula \eqref{eq:nabrr} that $\nabla \cdot
  \hat{\mu}(r) \geq 0$ whenever $\mathbf{Y}$ has no identical singular
  values. As argued below, this happens Lebesgue almost everywhere. Hence what remains is to
  show that the set of matrices with no identical singular values 
  and the estimator $\hat{\mu}(r)$ defined on this set fulfill the properties of Proposition
  \ref{prop:hausdorff}.  

Letting
$$E = \{ \mathbf{Y} \in M(p, q) \mid \mathbf{Y} \text{ has two identical singular values}
\}$$
it was shown by \cite{Mukherjee:2015} that $E$ is a proper subvariety, and
it is, in particular, a closed set.

For $\mathbf{Y} \in E^c$ with singular value decomposition
$$\mathbf{Y} = \sum_{k=1}^q d_k u_k v_k^T$$
and $d_1 > d_2 > \ldots > d_q \geq 0$, the reduced rank estimator is
given by
$$\hat{\mu}(r) = \sum_{k=1}^r d_k u_k v_k^T,$$
for $r \in \{1, \ldots, q - 1\}$. An application of e.g. Theorem 5.3 in \cite{Serre:2010} to $\mathbf{Y}^T
\mathbf{Y}$ shows that the singular values $d_1, \ldots, d_{q-1}$ as well as $u_1,
\ldots, u_{q-1}$ and $v_1, \ldots, v_{q-1}$ are $C^{\infty}$ on $E^c$, thus
$\hat{\mu}(r)$ is $C^{\infty}$ on $E^c$ for $r < q$. The algebra by
\cite{Mukherjee:2015} or \cite{Candes:2013} gives the actual
partial derivatives and thus the divergence of $\hat{\mu}(r)$. It does
require a bit of algebra, though, to realize that the divergence formula (9) in
\cite{Candes:2013} is identical to the formula in Theorem 3 in
\cite{Mukherjee:2015}, which is stated above as \eqref{eq:nabrr}. See
also \eqref{eq:nab} below.  

We complete the proof using Proposition \ref{prop:hausdorff} by
establishing that $\mathcal{H}^{pq - 1}(E) = 0$. To this end we show that the
codimension of $E$ is at least 2. 

Let $V_r(m)$ denote the Stiefel manifold of $r$-tuples of
$k$-dimensional orthonormal vectors in $\mathbb{R}^m$, and introduce
$h : V_{q-2}(q) \times V_q(p) \times [0,\infty)^{q-1} \to M(p, q)$ by 
$$h((v_k)_k, (u_k)_k, d) = \sum_{k=1}^{q-2} d_k
u_k v_k^T + d_{q-1}(u_{q-1}
v_{q-1}^T + u_q v_q^T),$$
where the unit vectors $v_{q-1}$ and $v_{q}$ are chosen (by $h$ and depending on
$(v_k)_k$) orthogonal to $v_1,
\ldots, v_{q-2}$. The set $E$ is
in the image of $h$ because when two singular values are identical the
singular value decomposition is not unique, and it is possible to
choose an orthogonal transformation such that
$v_{q-1}^T$ and $v_{q}^T$ in the singular value
decomposition are those chosen by the map $h$. The Stiefel manifold $V_r(m)$ has dimension
$$\mathrm{dim}(V_r(m)) = rm - \frac{1}{2}r(r + 1)$$
as a differentiable manifold. It follows that $E$ locally is contained
in the image under a Lipschitz map of a set of dimension 
$$q(q-2) - \frac{1}{2}(q-2)(q-1) + pq - \frac{1}{2}q(q+1) + q - 1 = pq
- 2.$$
It follows from Theorem 2.4.1 in \cite{Evans:1992} that $E$ has
Hausdorff dimension at most $pq - 2$, whence $\mathcal{H}^{pq-1}(E) =
0$. By Proposition \ref{prop:hausdorff} it follows that \eqref{eq:stein} holds, and the
divergence \eqref{eq:nabrr} is an unbiased
estimate of the degrees of freedom. 
\end{proof}

The fact that $E$ is a proper subvariety implies that $E$
  has codimension at least 1 and $E$ has Lebesgue measure zero. As
  argued above, this is not sufficient for Stein's lemma to hold. The
  set $E$ needs to be even smaller as expressed by the condition $\mathcal{H}^{pq-1}(E) =
0$ using the Hausdorff measure. Establishing that $E$ has codimension 2
is enough for this condition to be fulfilled. 

The argument above is closely related to the long
established fact that the set of matrices with repeated eigenvalues in
the set of real symmetric matrices has codimension 2. A result
credited to Neumann and Wigner, see p. 36 in \cite{Lax:2007}. It does, however,
appear somewhat complicated to determine the codimension of $E$ as a
subvariety, see \cite{Dana:2006} for the case of symmetric matrices,
and we proceeded in the argument above by effectively counting the free
parameters in the singular value decomposition instead.

Essentially the same argument as above can be used for spectral
function estimators provided that the spectral functions are suitably
well behaved. 

\begin{thm} \label{thm:specfun}
Consider a spectral function estimator
$$\hat{\mu} = \sum_{k=1}^q f_k(d_k) u_k v_k^T$$
with finite second moment and with spectral functions fulfilling that  $f_1, \ldots,
f_{q-1}$ are continuously differentiable on $(0, \infty)$, $f_q$ is
continuously differentiable on $[0, \infty)$, $f_k \geq f_l$ for $k
< l$ and $f_k' \geq 0$. Then $\mathrm{SURE}$
is an unbiased estimate of the risk, and 
\begin{equation} 
\label{eq:nab}
\nabla \cdot \hat{\mu} = (p - q) \sum_{k=1}^q \frac{f_k(d_k)}{d_k} +
\sum_{k=1}^q f_k'(d_k) + 2 \sum_{k,l=1 \atop k \neq l}^q
\frac{d_kf_k(d_k)}{d_k^2 - d_l^2}
\end{equation}
whenever $\mathbf{Y}$ has $q$ distinct singular values.
\end{thm}

\begin{proof} As written above, the proof is along the same lines as
  the proof of Theorem \ref{thm:rr}. The set $E$ has
  $(pq-1)$-dimensional Hausdorff measure zero, and on $E^c$ the
  estimator $\hat{\mu}$ is continuously differentiable -- under the
  differentiability assumptions on the spectral functions -- with divergence
  given by \eqref{eq:nab}. This divergence formula was shown by \cite{Candes:2013}
  and is given as (9) in their paper. 

  To use Proposition \ref{prop:hausdorff} we verify that $\nabla \cdot
  \hat{\mu} \geq 0$ on $E^c$. Clearly, as $f_k$ is positive and $f_k'$
  is also assumed positive, the first two terms in \eqref{eq:nab} are
  positive. For the third term we rearrange the double sum as 
$$\sum_{k,l=1 \atop k \neq l}^q
\frac{d_kf_k(d_k)}{d_k^2 - d_l^2} = \sum_{k=1}^q \sum_{l = k + 1}^q
\frac{d_kf_k(d_k) - d_lf_l(d_l)}{d_k^2 - d_l^2}.$$
Using that $d_k > d_l$ for $k < l$, and that this implies that 
$$f_k(d_k) \geq f_k(d_l) \geq f_l(d_l),$$
we have that for each term in this double sum
$$\frac{d_kf_k(d_k) - d_l f_l(d_l) }{d_k^2 - d_l^2} \geq
\frac{(d_k - d_l) f_l(d_l) }{d_k^2 - d_l^2} \geq 0.
$$
This completes the proof. 
\end{proof}

As stated in the proof above, formula \eqref{eq:nab} is identical to (9) from
\cite{Candes:2013}. An equivalent formula is given in Theorem 4 in
\cite{Mukherjee:2015}. Clearly, $f_k(d) = d1(k \leq r)$ is smooth,
whereas $f_k(d) = d 1(d \geq \lambda)$ is not, and Theorem
\ref{thm:specfun} doesn't apply to singular value hard thresholding. 

It's possible that the monotonicity requirements on the spectral
functions above can be relaxed,  and that \eqref{eq:momcond} can be
verified instead of the positivity on the divergence. But this won't be
pursued in this note. 

\section{Simulation}

This section presents the results from a simulation that illustrates the
unbiasedness of SURE for reduced rank estimation and singular value
soft thresholding, whereas \eqref{eq:nab} is shown to be a biased
estimate of degrees of freedom for singular value hard thresholding. 

For the simulation we made $B = 5000$ simulations with $p = q = 21$ and $Y_{ij}^b \sim \mathcal{N}(0,
1)$ for $b = 1, \ldots, B$. With 
$$\mathbf{Y}^b =  \sum_{k=1}^q d_k^b u_k^b (v_k^b)^T, \qquad d_1^b
\geq d_2^b \geq \ldots \geq d_q^b \geq 0$$ 
the singular value decomposition of the $b$th matrix $\mathbf{Y}^b =
(Y_{ij}^b)_{i,j}$ we computed the three estimators 
\begin{align*}
\hat{\mu}_b(r) & = \sum_{k=1}^r d_k^b u_k^b (v_k^b)^T 
                 \tag*{(\textrm{reduced rank})} \\
\overline{\mu}_b(r) & = \sum_{k=1}^q d_k^b1(d_k^b \geq \lambda) u_k^b (v_k^b)^T
                 \tag*{(\textrm{hard thresholding})} \\
\tilde{\mu}_b(\lambda) & = \sum_{k=1}^q (d_k^b - \lambda)_+ u_k^b (v_k^b)^T. 
                 \tag*{(\textrm{soft thresholding})}
\end{align*}

By the definition of the degrees of freedom, \eqref{eq:df}, the
estimate 
$$\hat{\mathrm{df}}_0(r) = \frac{1}{B} \sum_{b=1}^{B}
\mathrm{tr}(\hat{\mu}_b(r)^T (\mathbf{Y}^b - \mu))$$
is an unbiased estimate of $\mathrm{df}$ for the reduced rank
estimator, and similar estimates of $\mathrm{df}$ were computed for the other two
estimators. Note that such estimates based on the covariance
definition are of no use in real applications as they rely on knowledge of the
true mean. In this simulation the true mean was $\mu = 0$.

\begin{figure}
\makebox[\textwidth][c]{\includegraphics[width=16cm]{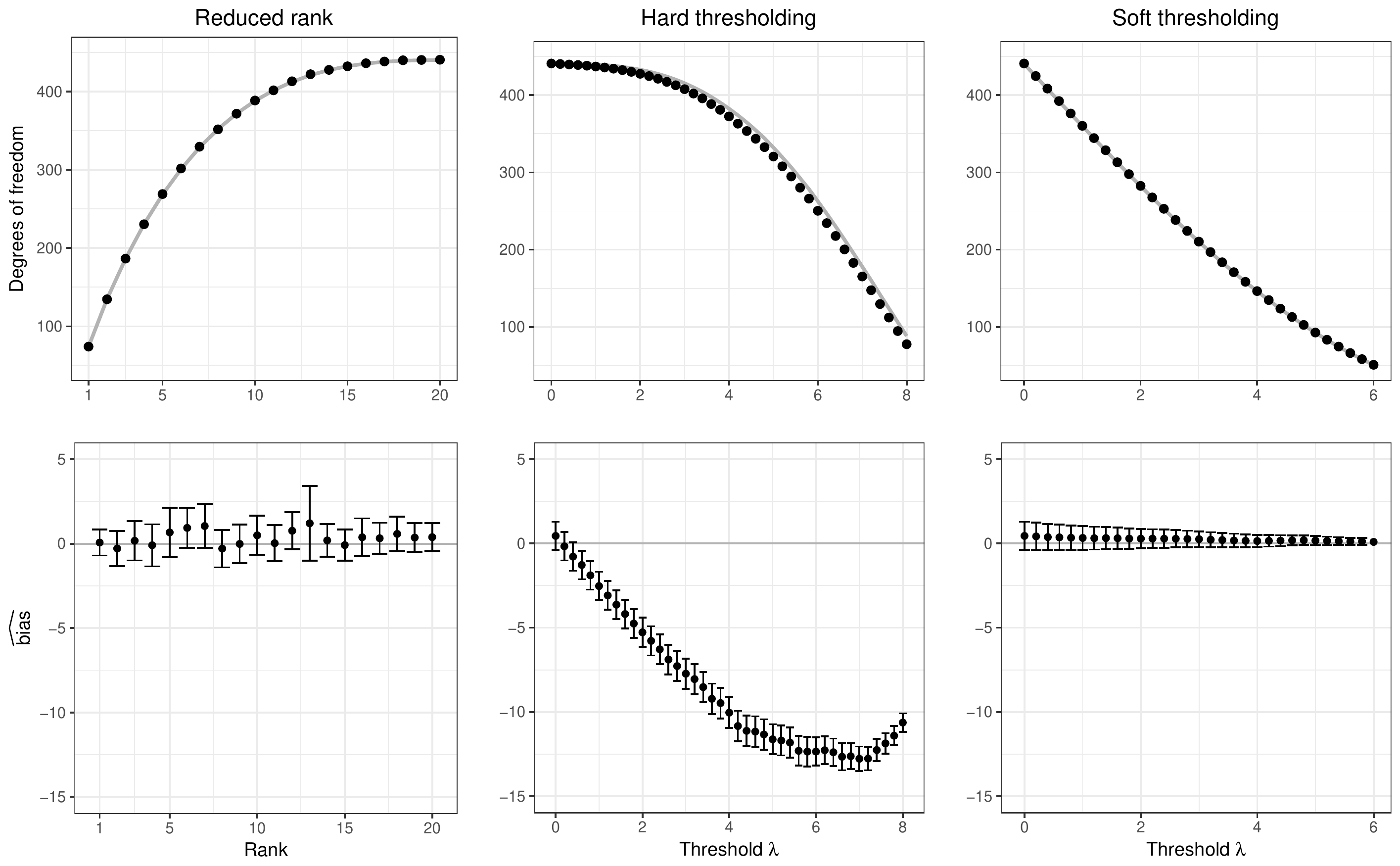}}
\caption{The top row shows average degrees of freedom as
  estimated by the divergence formulas (points) and by the covariance definition
  \eqref{eq:df} (gray line). The bottom row shows the bias of the
  divergence estimate for the three different estimators. For the
  bias, the 95\% confidence intervals shown quantify the simulation
  uncertainty. \label{fig:sim_results}} 
\end{figure}

Estimates based on the divergence were computed for each of the three
estimators as follows
\begin{align*}
\hat{\mathrm{df}}(r) & = pr + \frac{1}{B} \sum_{b=1}^B  \sum_{k=1}^r \sum_{l= r + 1}^q
\frac{(d_k^b)^2 + (d_l^b)^2}{(d_k^b)^2 - (d_l^b)^2} \\
\overline{\mathrm{df}}(\lambda) & = \frac{1}{B} \sum_{b=1}^B \left((p
                                  - q + 1) \sum_{k=1}^q 1(d_k^b \geq \lambda)
+ 2 \sum_{k,l=1 \atop k \neq l}^q
\frac{(d_k^b)^21(d_k^b \geq \lambda)}{(d_k^b)^2 - (d_l^b)^2} \right)\\
\tilde{\mathrm{df}}(\lambda) & = \frac{1}{B} \sum_{b=1}^B \left((p -
                               q) \sum_{k=1}^q \left(1 - \frac{\lambda}{d_k^b}\right)_+ +
\sum_{k=1}^q 1(d_k^b \geq \lambda) + 2 \sum_{k,l=1 \atop k \neq l}^q
\frac{d_k^b(d_k^b - \lambda)_+}{(d_k^b)^2 - (d_l^b)^2} \right).
\end{align*}
Note that with 
$$r(\lambda) =  \sum_{k=1}^q 1(d_k^b \geq \lambda)$$
it holds that $\hat{\mathrm{df}}(r(\lambda)) = \overline{\mathrm{df}}(\lambda)$.

The bias estimate is defined as
$$\widehat{\mathrm{bias}}(r) = \hat{\mathrm{df}}(r) - \hat{\mathrm{df}}_0(r)$$
for the reduced rank estimator and likewise for the other two
estimators. For reduced rank and soft thresholding, the bias is zero
by the theoretical results. 

Figure \ref{fig:sim_results} shows the results of the simulation. It
shows clearly that for singular value hard thresholding the bias is
non-zero, though it is relatively small. The results for reduced rank
and soft thresholding are completely in concordance with the theoretical results
that the bias is 0, and with the simulation results presented by \cite{Candes:2013}
and \cite{Mukherjee:2015}.

\section{Final comments}

There is no claim of originality in terms of the estimators considered
or the formulas presented for estimating degrees of
freedom and computing SURE. These can be found in the papers by \cite{Candes:2013}
and by \cite{Mukherjee:2015}. However, neither of these two papers -- nor
other papers that the author is aware of -- gives a complete proof of the
fact that Stein's lemma does apply to the reduced rank estimator. The
purpose of this note was to give this proof. 

We may note that the reduced rank estimator, $\hat{\mu}(r)$, and the
hard thresholding estimator, $\overline{\mu}(\lambda)$, provide the
exact same sequence of estimators for a given observation when viewed
as functions of $r$ and $\lambda$, respectively. Yet, for a fixed $r$
we can by SURE obtain an unbiased risk estimate for  $\hat{\mu}(r)$,
while for fixed $\lambda$, the SURE formula based on the divergence
estimate of degrees of freedom doesn't give an unbiased risk estimate
for $\overline{\mu}(r)$. While this is understandable when the mapping
between the two sequences of estimators is data dependent, it also
highlights that the parametrization matters when estimators are
assessed via their frequentistic risk. When tuning parameters are selected by
minimizing a risk estimate, this leads to the somewhat peculiar 
phenomenon that different parametrizations can lead to different
choices of tuning parameters. Or as is the case here, that one
parametrization provides an unbiased risk estimate, while another
provides a biased risk estimate, even though the risk estimates are
identical.

\bibliographystyle{agsm}

\end{document}